\documentclass{amsart}
\usepackage{amsmath}
\usepackage{amssymb}

\newtheorem{prop}{Proposition}
\newtheorem{theo}[prop]{Theorem}
\newtheorem{lemm}[prop]{Lemma}
\newtheorem{cor}[prop]{Corollary}

\theoremstyle{remark}

\newtheorem*{exam}{Example}

\newcommand{\wLambda}{\widetilde{\Lambda}}
\newcommand{\alg}{\mathrm{alg}}
\newcommand{\ra}{\rightarrow}

\newcommand{\Pic}{\mathrm{Pic}}

\newcommand{\bC}{\mathbb C}

\newcommand{\bP}{\mathbb P}
\newcommand{\bQ}{\mathbb Q}
\newcommand{\bR}{\mathbb R}
\newcommand{\bZ}{\mathbb Z}

\newcommand{\cC}{\mathcal C}

\newcommand{\cO}{\mathcal O}

\newcommand{\cX}{\mathcal X}

\newcommand{\cZ}{\mathcal Z}

\newcommand{\mfM}{\mathfrak M}

\author{Arend Bayer}
\address{University of Edinburgh, School of Mathematics and Maxwell Institute,
James Clerk Maxwell Building,
The King's Buildings,
Mayfield Road,
Edinburgh,
Scotland EH9 3JZ}
\email{arend.bayer@ed.ac.uk}

\author{Brendan Hassett}
\address{Department of Mathematics, Rice University, MS 136,
Houston, Texas 77251-1892, USA}
\email{hassett@rice.edu}

\author{Yuri Tschinkel}
\address{Courant Institute, New York University, New York, NY 11012, USA}
\address{Simons Foundation, 160 Fifth Avenue, New York, NY 10010, USA}
\email{tschinkel@cims.nyu.edu}

\title{Mori cones of holomorphic symplectic varieties of K3 type}

\begin{document}

\begin{abstract}
We determine the Mori cone of holomorphic symplectic varieties deformation
equivalent to the punctual Hilbert scheme on a K3 surface. Our description is given in terms of
Markman's extended Hodge lattice.
\end{abstract}


\maketitle

\section*{Introduction}

Let $X$ be an irreducible holomorphic symplectic manifold.
Let $\left(,\right)$ denote the Beauville-Bogomolov form on $H^2(X,\bZ)$;
we may embed $H^2(X,\bZ)$ in
$H_2(X,\bZ)$ via this form.  
Fix a polarization $h$ on $X$; by a fundamental result of Huybrechts \cite{HuyBasic},
$X$ is projective if it admits a divisor class $H$ with $\left(H,H\right)>0$. 
It is expected that finer birational properties 
of $X$ are also encoded by the Beauville-Bogomolov form and the Hodge structure
on $H^2(X)$, along with appropriate extension data.
In particular, natural cones appearing in the minimal 
model program---the moving cone, the nef cone, the pseudo-effective cone---should
have a description in terms of this form.

Now assume $X$ is deformation equivalent to 
the punctual Hilbert scheme $S^{[n]}$ of a K3 surface $S$ with $n>1$.  
Recall that 
\begin{equation} \label{eqn:H2}
H^2(S^{[n]},\bZ)_{\left(,\right)}=H^2(S,\bZ) \oplus_{\perp} \bZ \delta,
\quad \left(\delta,\delta\right)=-2(n-1)
\end{equation}
where the restriction of the Beauville-Bogomolov form to the first factor
is just the intersection form on $S$, and $2\delta$ is the class of
the locus of non-reduced subschemes. 
Recall from \cite{Kovacs} that for K3 surfaces $S$, the
cone of (pseudo-)effective divisors is the closed cone generated by
$$\{D \in \Pic(S): \left(D,D\right) \ge -2, \left(D,h\right)>0\}.$$
The first attempt to extend this to higher dimensions
was \cite{HT1}.  Further work on moving cones was presented in \cite{HT2,MarkMov},
which built on Markman's analysis of monodromy groups.  
The characterization of extremal rays arising from Lagrangian projective
spaces $\bP^n \hookrightarrow X$ has been addressed in \cite{HT2,HHT3}
and \cite{BJ}.  The paper \cite{HT4} proposed a general framework describing
all types of extremal rays; however, Markman found counterexamples in dimensions $\ge 10$,
presented in \cite{BM1}.  

The formalism of Bridgeland stability conditions \cite{Br1,Br2} has led to breakthroughs
in the birational geometry of moduli spaces of sheaves on surfaces.  The case of punctual Hilbert
schemes of $\bP^2$ and del Pezzo surfaces was investigated
by Arcara, Bertram, Coskun, and Huizenga \cite{ABCH,Hui, BC, CH}.
The effective cone on $(\bP^2)^{[n]}$ has a beautiful and complex structure as $n$
increases, which only becomes transparent in the language of stability conditions.  
Bayer and Macr{\`{\i}} resolved the case of punctual Hilbert schemes and
more general moduli spaces of sheaves on K3 surfaces \cite{BM1,BM2}.  Abelian surfaces,
whose moduli spaces of sheaves include generalized Kummer varieties, have been studied as well
\cite{YY,Yosh}.

In this note, we extend the results obtained for moduli spaces of sheaves over K3 surfaces
to all holomorphic symplectic manifolds arising as deformations of punctual Hilbert schemes of
K3 surfaces.  Our principal result is Theorem~\ref{theo:main} below, providing a description of the Mori
cone (and thus dually of the nef cone). 

In any given situation, this also leads to an effective method to determine the list of marked
minimal models (i.e., birational maps $f \colon X \dashrightarrow Y$ where $Y$ is also a
holomorphic symplectic manifold): the movable cone has been described by Markman \cite[Lemma
6.22]{MarkSurv}; by \cite{HT2}, it admits a wall-and-chamber decomposition whose walls are the orthogonal
complements of extremal curves on birational models, and whose closed chambers corresponds
one-to-one to marked minimal model, as the pull-backs of the corresponding nef cones.

\subsection*{Acknowledgments:}   
The first author was supported by NSF grant 1101377;
the second author was supported by NSF
grants 0901645, 0968349, and 1148609; the third  
author was supported by NSF grants 
0968318 and 1160859.
We are grateful to Emanuele Macr{\`{\i}} for helpful conversations, to 
Eyal Markman for constructive criticism and correspondence, to
Claire Voisin for helpful comments on deformation-theoretic arguments in a draft of this paper,
and to Ekatarina Amerik for discussions on holomorphic symplectic contractions.
We are indebted to the referees for their careful reading of our manuscript. 
The first author would also like to thank Giovanni Mongardi for discussions and a preliminary
version of \cite{Mon}. Related questions for general hyperk\"ahler manifolds have been treated
in \cite{VerAmerik}.

\section{Statement of Results}
\label{sect:statements}

Let $X$ be deformation equivalent to the Hilbert scheme of length-$n$
subschemes of a K3 surface. 
Markman \cite[Cor.~9.5]{MarkSurv} describes an extension of lattices 
$$H^2(X,\bZ) \subset \wLambda$$
and weight-two Hodge structures
$$H^2(X,\bC) \subset \wLambda_{\bC}$$
characterized as follows:  
\begin{itemize}
\item{the orthogonal complement of $H^2(X, \bZ)$ in $\wLambda$ has rank one, and is generated by
a primitive vector of square $2n-2$;}
\item{as a lattice
$$\wLambda \simeq U^4 \oplus (-E_8)^2$$
where $U$ is the hyperbolic lattice and $E_8$ is the positive definite
lattice associated with the corresponding Dynkin diagram;}
\item{any parallel transport operator $H^2(X, \cZ) \to H^2(X', \cZ)$ naturally lifts to a Hodge
isometry $\wLambda_X \to \wLambda_{X'}$; the induced action of the monodromy group on $\wLambda/H^2(X,\bZ)$
is encoded by a character $cov$ (see \cite[Sec.~4.1]{MarkJAG});}
\item{we have the following Torelli-type statement:  $X_1$ and $X_2$ are
birational if and only if there is Hodge isometry
$$\wLambda_1 \simeq \wLambda_2$$
taking $H^2(X_1,\bZ)$ isomorphically to $H^2(X_2,\bZ)$;}
\item{if $X$ is a moduli space $M_v(S)$ of sheaves (or of Bridgeland-stable
complexes) over a K3 surface $S$
with Mukai vector $v$ then there is an isomorphism from $\wLambda$ to the
Mukai lattice of $S$ taking 
$H^2(X,\bZ)$ to $v^{\perp}$.}
\end{itemize}
Generally, we use $v$ to denote a primitive generator for the orthogonal complement
of $H^2(X,\bZ)$ in $\wLambda$.  Note that $v^2=\left(v,v\right) = 2n-2$.
When $X \simeq M_v(S)$ we may take the Mukai vector $v$ as the generator.  

There is a canonical homomorphism
$$\theta^{\vee}\colon \wLambda \twoheadrightarrow H_2(X,\bZ)$$
which restricts to an inclusion
$$H^2(X,\bZ) \subset H_2(X,\bZ)$$
of finite index.  
By extension, it induces a $\bQ$-valued Beauville-Bogomolov form on
$H_2(X, \bZ)$.

Assume $X$ is projective.  
Let $H^2(X)_{alg} \subset H^2(X,\bZ)$ and $\wLambda_{\alg}\subset \wLambda$ denote
the algebraic classes, i.e., the integral classes of type $(1,1)$.  
The Beauville-Bogomolov form on $H^2(X)_{alg}$ has signature $(1,\rho(X)-1)$,
where $\rho(X)=\dim(H^2_{alg}(X)).$
The {\em Mori cone} of $X$ is defined as the closed cone in $H_2(X,\bR)_{alg}$
containing the classes of algebraic curves in $X$.  The {\em positive cone} (or more accurately,
non-negative cone) in
$H^2(X,\bR)_{alg}$ is the closure of the connected component of the cone
$$\{D \in H^2(X,\bR)_{alg}: D^2 > 0 \}$$
containing an ample class. 
The dual of the positive cone in $H^2(X,\bR)_{alg}$ is the 
positive cone.

\begin{theo} \label{theo:main}
Let $(X,h)$ be a polarized holomorphic symplectic manifold as above.
The Mori cone in $H_2(X,\bR)_{alg}$ is generated by classes
in the positive cone 
and the images under $\theta^{\vee}$ of the following:
$$\{a \in \wLambda_{\alg}: a^2 \ge -2, \left| \left(a,v\right)\right| \le v^2/2,
\left(h,\theta^\vee(a)\right)>0\}.$$
\end{theo}

This generalizes \cite[Theorem 12.2]{BM2}, which treated the case of moduli spaces of sheaves on K3
surfaces.
This allows us to compute the full nef cone of $X$ from its Hodge structure
once a single ample divisor is given.  
As another application of our methods, we can bound the length of extremal rays of the Mori cone with respect to Beauville-Bogomolov pairing:
\begin{prop} \label{prop:length}
Let $X$ be a projective holomorphic symplectic manifold as above.
Then any extremal ray of its Mori cone contains an effective curve class $R$ with
\[ 
(R, R) \ge -\frac{n+3}2. \]
\end{prop}
The value $-\frac{n+3}2$ had been conjectured in \cite{HT4}.
Proposition \ref{prop:length} has been obtained independently by Mongardi \cite{Mon}. His proof is
based on twistor deformations, and also applies to non-projective manifolds.

\section{Deforming extremal rational curves}
In this section, we consider general irreducible holomorphic symplectic manifolds, not necessarily of K3 type.  
Our arguments are based on the deformation theory of rational curves on holomorphic symplectic
manifolds, as first studied in \cite{Ran}.
Recall the definition of a {\em parallel transport operator} $\phi:H^2(X,\bZ) \ra H^2(X',\bZ)$ between
manifolds of a fixed deformation type:  
there is a smooth proper family $\pi:\cX \ra B$ over a connected
analytic space, points $b,b'\in B$ with $\cX_b:=\pi^{-1}(b)\simeq X$ and $\cX_{b'}\simeq X'$,
and a continuous path $\gamma:[0,1]\ra B,\gamma(0)=b,\gamma(1)=b'$, such that parallel transport
along $\gamma$ induces $\phi$. 

\begin{prop} \label{prop:deform}
Let $X$ be a projective holomorphic symplectic variety and $R$ the class of an extremal rational
curve $\bP^1 \subset X$ with $\left(R,R\right)<0$.  Suppose that $X'$ is deformation
equivalent to $X$ and $\phi:H^2(X,\bZ) \ra H^2(X',\bZ)$ is a parallel transport
operator associated with some family.
If $R':=\phi(R)$ is a Hodge class, and if there exists a K\"ahler class $\kappa$ on $X'$ with
$\kappa.R' > 0$, then a multiple of $R'$ is effective and
represented by a cycle of rational curves.
\end{prop}
Note that $X'$ need not be projective here.  
\begin{proof}
Fix a proper holomorphic family $\pi:\cX \ra B$ over an irreducible
analytic space $B$ with $X=\cX_b$.
We claim there exists a rational curve $\xi: \bP^1\ra X$ 
with class $[\xi(\bP^1)]\in \bQ_{\ge 0}R$ satisfying the following property:
for each $b''$ near $b$ such that $R$ remains algebraic there exists a 
deformation $\xi_{b''}:\bP^1 \ra \cX_{b''}$ of $\xi$.

Let $\omega$ denote the holomorphic symplectic form on $X$, $f:X\ra Y$ the birational contraction associated
with $R$, $E$ an irreducible component of the
exceptional locus of $f$, $Z$ its image in $Y$, and $F$ a generic fiber of $E\ra Z$. 
We recall structural results about the contraction $f$:
\begin{itemize}
\item{$\omega$ restricts to zero on $F$ \cite[Lemma~2.7]{Kaledin};}
\item{the smooth locus of $Z$ is symplectic with two-form pulling back
        to $\omega|E$ \cite[Thm.~2.5]{Kaledin} \cite[Prop.~1.6]{Namikawa};}
\item{the dimension $r$ of $F$ equals the codimension of $E$ \cite[Thm.~1.2]{Wierzba}.}
\end{itemize}
Second, we review general results about rational curves $\xi:\bP^1 \ra X$:
\begin{itemize}
\item{a non-constant morphism $\xi:\bP^1 \ra X$ deforms in at least a $(2n+1)$-dimensional family \cite[Cor.~5.1]{Ran};}
\item{the fibers of $E\ra Z$ are rationally chain connected \cite[Cor.~1.6]{HacMcK};}
\item{a non-constant morphism $\xi:\bP^1 \ra F$ deforms in at least a $(2r+1)$-dimensional family \cite[Thm.~1.2]{Wierzba}.}
\end{itemize}

Let $\xi:\bP^1
\to F\subset X$ be a rational curve of minimal degree passing through the generic point of $F$.  
We do not assume {\em a priori} that $F$ is smooth.  The normal bundle $N_{\xi}$
was determined completely in \cite[\S 9]{CMSB}, which gives a precise
classification of $F$.  The fact that rational curves in $F$ deform in $(2r-2)$-dimensional families 
implies that every rational curve through the generic point of $F$ is doubly dominant,
i.e., it passes through {\em two} generic points of $F$.
Using a bend-and-break argument \cite[Thm.~2.8 and 4.2]{CMSB}, we may conclude that the normalization
of $F$ is isomorphic to $\bP^r$.  Note that the generic $\xi:\bP^1 \to F$ is immersed in $X$
by \cite[\S 3]{Kebekus}.  

Using standard exact sequences for normal bundles
and the fact that 
$\xi:\bP^1 \to F$ is immersed in $X$, 
one sees that (cf.~\cite[Lemma 9.4]{CMSB}) 
$$N_{\xi}
        \simeq \cO_{\bP^1}(-2) \oplus \cO_{\bP^1}(-1)^{r-1} \oplus \cO_{\bP^1}^{2(n-r)} \oplus \cO_{\bP^1}(1)^{r-1}.$$

The crucial point is that $h^1(N_{\xi})=1$.  Thus we may
apply \cite[Cor.~3.2]{Ran} to deduce that 
the deformation space of $\xi(\bP^1)\subset X$ has dimension $2n-2$;
\cite[Cor.~3.3]{Ran} then implies that $\xi(\bP^1)$ persists in deformations of $X$ for which $R$ remains a Hodge class.
This proves our claim.  

\begin{exam}
The extremality assumption is essential, as shown by an example suggested by Voisin:  Let $S$ be a K3 surface
arising as a double cover of $\bP^1\times \bP^1$ branched over a curve of bidegree $(4,4)$ and $X=S^{[2]}$.
We may regard $\bP^1 \times \bP^1 \subset X$ as a Lagrangian surface.  Consider a smooth curve
$C \subset \bP^1 \times \bP^1 \subset X$
of bidegree $(1,1)$.  The curve $C$ persists only in the codimension-{\em two} subspace of the deformation space of $X$
where $\bP^1 \times \bP^1$ deforms (see \cite{Voisin});
note that $N_{C/X} \simeq \cO_{\bP^1}(2) \oplus \cO_{\bP^1}(-2)^2$.
\end{exam}

We return to the proof of Proposition~\ref{prop:deform}.  
Consider the relative Douady space parametrizing rational curves of class $[\xi(\bP^1)]$
in fibers of $\cX \ra B$ and their specializations.  
Remmert's Proper Mapping theorem \cite[Satz~23]{Remmert}
implies that its image $B_R \subset B$ is proper and that over each $b' \in B_R$ there exists a 
cycle of rational curves in $\cX_b'$ that is a specialization
of $\xi_{b''}(\bP^1)$.    

To prove the Proposition~\ref{prop:deform}, we need to produce a family $\varpi:\cX^+ \ra B^+$ over an 
irreducible base, with both $X$ and $X'$
as fibers, such that $X'$ lies over a point of $B^+_R$ and $R'=\phi(R)$ coincides
with $\phi^+(R)$.  Here $\phi^+$ is the parallel transport mapping associated with $\varpi$.
Then the Proper Mapping theorem would guarantee that $R'$ is in the Mori cone of $X'$.

\begin{lemm} \label{lemm:family}
Let $X, X', R$ be as in Proposition \ref{prop:deform}.
There exists a smooth proper family $\varpi\colon \cX^+ \ra B^+$ over an irreducible analytic space,
points $b,b' \in B^+$ with $\cX^+_b\simeq X$ and $\cX^+_{b'}\simeq X'$, and a section
$$\rho\colon B^+ \ra \bR^2\varpi_*\bZ$$
of type $(1,1)$, such that $\rho(b)=R$ and $\rho(b')=R'$.
\end{lemm}
\begin{proof}
This proof is essentially the same as the argument for Proposition 5.12 of \cite{MarkMov}.  We summarize the key points.

Let $\mfM$ denote the moduli space of marked holomorphic symplectic manifolds of
K3 type \cite[Sec.~1]{HuyBasic}.  Essentially, this is obtained by gluing together all the
local Kuranishi spaces of the relevant manifolds.  It is non-Hausdorff.
Let $\mfM^{\circ}$ denote a connected component of $\mfM$ containing $X$ equipped with a suitable marking.

Consider the subspace $\mfM^{\circ}_R$ such that $R$ is type $(1,1)$ and $\kappa. R>0$
for {\em some} K\"ahler class, which may vary from point to point of the moduli space.  
This coincides with an open subset of the 
preimage of the hyperplane $R^{\perp}$ under the period map $P$ \cite[Claim~5.9]{MarkMov}.  
Furthermore, for general periods $\tau$---those for which $R$ is the unique integral class
of type $(1,1)$---the preimage $P^{-1}(\tau)$ consists of a single marked manifold
\cite[Cor.~5.10]{MarkMov}.  The proof of this in \cite{MarkMov} only requires that $\left(R,R\right)<0$.  
(The Torelli Theorem implies two manifolds share the same period point only if they are bimeromorphic \cite[Th.~1.2]{MarkSurv},
but if $R$ is the only algebraic class, the only other bimeromorphic model would not admit a
K\"ahler class $\kappa'$ with $\kappa'.R > 0$.)
Finally, $\mfM^{\circ}_R$ is path-connected by \cite[Cor.~5.11]{MarkMov}.  

Choose a path $\gamma:[0,1] \ra \mfM^{\circ}_R$ joining $X$ and $X'$ equipped with suitable markings, taking $R$ and $R'$ to the 
distinguished element $R$ in the reference lattice.  Cover the image with a finite number of small connected neighborhoods $U_i$ admitting
Kuranishi families.  We claim there exists an analytic space $B^+$
$$\gamma([0,1]) \subset B^+ \subset \cup_{i=1}^m U_i$$
with a universal family.  Indeed, we choose $B^+$ to be an open neighborhood 
of $\gamma([0,1])$ admitting a deformation
retract onto the path, but small enough so it is contained in the union of the $U_i$'s.
The topological triviality of $B^+$ means there is no obstruction to gluing local families.  
\end{proof}
This completes the proof of Proposition~\ref{prop:deform}. 
\end{proof}

\section{Proof of Theorem \ref{theo:main}}

In the case where $X = M_v(S)$ is a smooth moduli space of Gieseker-stable sheaves (or, indeed, of
Bridgeland-stable objects) on a K3 surface $S$, the statement is proven in \cite[Theorem 12.2]{BM2}. 
We will prove Theorem \ref{theo:main} by reduction to this case. 

The key argument is based on important results of Markman on the cone of movable divisors and its
relation to the monodromy group. Let $\cC_{\mathrm{mov}}^o$ be the intersection of the movable
cone with the positive cone in $H^2(X, \bR)_{\mathrm{alg}}$. Each wall of the movable cone
corresponds to a divisorial contraction of an irreducible exceptional divisor $E$ on some
birational model of $X$; the wall is contained
in the orthogonal complement $E^\perp$ of $E$ with respect to the Beauville-Bogomolov form. 

\begin{theo}[Markman]
\label{thm:Weylgroup}
\begin{enumerate}
\item
Let $X$ be an irreducible holomorphic symplectic manifold. 
Consider the reflection $\rho_E \colon H^2(X, \bR) \to H^2(X, \bR)$ that leaves $E^\perp$ fixed and sends
$E$ to $-E$. Then $\rho_E$ is defined over $\bZ$, acts by a monodromy transformation, and
extends to a Hodge isometry of the extended lattice $H^2(X) \subset \wLambda$.
\item
Let $W_{\mathrm{Exc}}$ be the Weyl group generated by reflections $\rho_E$ for all irreducible exceptional
divisors $E$ on all marked birational models of $X$. Then
$\cC_{\mathrm{mov}}^o$ is a fundamental domain of the action of $W_{\mathrm{Exc}}$ on the positive cone.
\end{enumerate}
\end{theo}
\begin{proof}
These results are reviewed in \cite[Section 6]{MarkSurv}. The first statement was originally proved
in \cite[Corollary 3.6]{MarkMov}. The second statement is \cite[Lemma 6.22]{MarkSurv}. (Note that
the definition of $W_{\mathrm{Exc}}$ in \cite[Definition 6.8]{MarkSurv} is slightly different to the
one given above; by \cite[Theorem 6.18, part (3)]{MarkSurv} they are equivalent.)
\end{proof}

\begin{cor} \label{corweylgroup}
Let $R \in H_2(X)$ be an algebraic class with $\left(R, R\right) < 0$. Then there exists a birational model
$X'$ of $X$, and a parallel transport operator
$\psi \colon H^2(X) \to H^2(X')$ such that one of the two following conditions hold:
\begin{enumerate}
\item $\psi(R)$ generates an extremal ray of the Mori cone. \label{enum:extreme}
\item Neither $\psi(R)$ nor $-\psi(R)$ is in the Mori cone. \label{enum:noneffective}
\end{enumerate}
In either case $X'$ admits a K\"ahler class $\kappa$ with $\kappa.\psi(R)>0$.  
\end{cor}
Note that $\psi$ may be non-trivial even when $X = X'$. 
\begin{proof}
The statement immediately follows from the following claim:
\emph{There exists $X', \psi$ such that the orthogonal complement $\psi(R)^\perp$ intersects
the nef cone in full dimension, and such that there exists an ample class $h$ with
$h.\psi(R) > 0$.}
Case \eqref{enum:extreme} corresponds to the case that $\psi(R)^\perp$ contains a wall of the nef
cone, and case \eqref{enum:noneffective} to the case that $\psi(R)^\perp$ intersects the interior.
Either way, we have a K\"ahler class $\kappa$ meeting $\psi(R)$ positively.

We first proof the claim with ``nef cone'' replaced by ``movable cone'' and
``ample class'' by ``movable class''. Since $(R, R) < 0$, the orthogonal complement $R^\perp$
intersects the positive cone; therefore, we can use the Weyl group action of $W_{\mathrm{Exc}}$
to force the intersection of $\psi(R)^\perp$ and the movable cone to be full-dimensional. In case
$\psi(R)^\perp$ contains a wall of the movable cone, $R$ is proportional to an irreducible
exceptional divisor $E^\perp$, and the reflection $\rho_E$ at $E$ can be used to ensure the second
condition.

Now we use the chamber decomposition of the movable cone, whose chambers are given by pull-backs
of nef cones of marked birational models (see \cite{HT2}): at least one of the closed chambers
 intersects $\psi(R)^\perp$ in full dimension, such that the interior lies on the side with positive
intersection with $\psi(R)$. The identification of $H^2$ of different birational models is
induced by a parallel transport operator. 
\end{proof}

To prove Theorem \ref{theo:main}, we will use the following facts:
\begin{itemize}
\item By assumption, there exists a deformation of $X$ to a Hilbert scheme $S^{[n]}$ of a projective K3
surface $S$; by the surjectivity of the Torelli map for K3 surface, we may further deform $S$ such 
that a given class in $\wLambda_X$ becomes algebraic in $H^*(S) \cong \wLambda_{S^{[n]}}$.
\item By \cite[Theorem 12.2]{BM2}, the main theorem holds for any moduli space $M_\sigma(v)$ of
$\sigma$-stable objects of given primitive Mukai vector $v$ on any projective K3 surface
(in particular, for any Hilbert scheme).
\item By \cite[Theorem 1.2]{BM2}, any birational model of $M_\sigma(v)$ is also a moduli space of
stable objects (with respect to a different stability condition), and in particular the main Theorem
holds.
\end{itemize}

We will prove Theorem \ref{theo:main} by deformation to the Hilbert scheme $X'$, followed by a second
deformation to a birational model $X''$ of $X'$ using Corollary \ref{corweylgroup}. By abuse of
notation, we will use the same letters $\phi, \psi$ to denote the parallel transport operators 
on $H^2, H_2$ and $\wLambda$ for the deformations from $X$ to $X'$, and from $X'$ to $X''$,
respectively. 

We first prove that the Mori cone of $(X,h)$ is contained in the cone described in Theorem
\ref{theo:main}.  Let $R$ be a generator of one of its extremal rays. 
Let $X'$ be a deformation-equivalent Hilbert scheme with parallel transport operator
$\phi$ that $\phi(R)$ is algebraic. 
We apply Corollary \ref{corweylgroup} to $\phi(R)$; thus there exists 
a birational model $X''$ of $X'$ such that $\psi \circ \phi(R)$ satisfies property
\eqref{enum:extreme} or
\eqref{enum:noneffective} as stated in the Corollary. By Proposition
\ref{prop:deform}, $\psi \circ \phi(R)$ is effective, excluding case \eqref{enum:noneffective};
thus $\psi \circ \phi(R)$ is extremal on $X''$. Since $X''$ is a moduli space of stable 
objects on a K3 surface, it is of the form $\theta^\vee(a)$ with $a$ as stated in the Theorem.
Since the Mori cone is generated by the positive cone and its extremal rays, this proves the claim.

Conversely, consider a class $R = \theta^\vee(a)$ where $a \in \wLambda_{X, \mathrm{alg}}$ satisfies
the assumptions in the Theorem. We may assume $(R, R) < 0$. Again we deform to a Hilbert scheme
$X'$ such that $\phi(R)$ (or, equivalently, $\phi(a)$) is algebraic, and 
apply Corollary \ref{corweylgroup} to $\phi(R)$. 
Let $R'':= \psi \circ \phi(R) \in H_2(X'') $ and $a'' := \psi \circ \phi(a) \in \wLambda_{X''}$ be the 
corresponding classes. By Theorem \cite[Theorem
12.2]{BM2}, the class $R''$ is effective; by the conclusion of the Corollary, it has to be extremal.
Thus we can apply Proposition \ref{prop:deform} to $R''$, and conclude that $R$ is effective.

This finishes the proof of Theorem \ref{theo:main}.

\begin{proof}[Proof of Proposition \ref{prop:length}]
In the case of moduli spaces of sheaves or Bridgeland-stable objects on a projective K3 surfaces,
the statement is proved in \cite[Proposition 12.6]{BM2}.
By the previous argument, there is a family $\pi: \cX \to B$ such that $\cX_{b_0}$ is a moduli space of
sheaves on a K3 surface with $[R]$ extremal, and such that $\cX_{b_1} \cong X$.
By \cite[Theorem 1.2]{BM2}, there exists a wall in the space of Bridgeland stability conditions
contracting $R$.  Let $R_0$ be the rational curve on $\cX_{b_0}$ in the ray $\bR_{\ge 0}[R]$ with
$(R_0, R_0) \ge -\frac{n+3}2$ given by \cite[Proposition 12.6]{BM2}.  The curve $R_0$ is a minimal
free curve in a generic fibre of the exceptional locus over $B$ (see \cite[Section 14]{BM2});
therefore, the deformation argument in Proposition \ref{prop:deform} applies directly to $R_0$
(rather than a multiple) and implies the conclusion.  \end{proof}

\bibliographystyle{alpha}
\bibliography{Moricones}
\end{document}